\documentclass[11pt,letterpaper]{amsart}
\usepackage{amsmath,amssymb,a4}
\usepackage{enumerate}
\usepackage[usenames]{color}
\usepackage{bbold}
\theoremstyle{plain}
\newtheorem{theorem}{Theorem}[section]
\newtheorem{lemma}[theorem]{Lemma}
\newtheorem{corollary}[theorem]{Corollary}
\newtheorem{proposition}[theorem]{Proposition}
\theoremstyle{definition}
\newtheorem{remark}[theorem]{Remark}

\newtheorem{definition}[theorem]{Definition}

\newcommand{\norm}[1]{\left\lVert#1\right\rVert}

\newcommand{\parenthesis}[1]{\left( #1 \right)}

\numberwithin{equation}{section}

\def\loc{\operatorname{loc}}

\def\esssup{\operatornamewithlimits{ess\,sup}}

\def\sgn{\operatorname{sgn}}

\def\R{\mathbb R}

\begin{document}

\title{On weak-type $(1,\,1)$ for averaging type operators} 

\author[S. Baena-Miret]{Sergi Baena-Miret$^{**}$}
\address{Departament de Matem\`atiques i Inform\`atica, Universitat de Barcelona, 08007 Barcelona, Spain.} 
\email{sergibaena@ub.edu}

\author[M.J. Carro]{Mar\'{\i}a J. Carro$^*$}

\address{Departamento de An\'alisis Matem\'atico y Matem\'atica Aplicada, Universidad Complutense de Madrid,  Spain.} 
\email{mjcarro@ucm.es}

\thanks{The authors  were partially supported by grant PID2020-113048GB-I00 funded by MCIN/AEI/ 10.13039/501100011033.\\
\indent\emph{E-mail addresses:} $^{*}$\texttt{mjcarro@ucm.es}, $^{**}$\texttt{sergibaena@ub.edu}}



\begin{abstract} It is known that, due to the fact that $L^{1, \infty}$ is not a Banach space,  if $(T_j)_j$ is a sequence of bounded operators so that
$$
  T_j:L^1\longrightarrow L^{1, \infty},
  $$
with norm less than or equal to $||T_j||$ and $\sum_j ||T_j||<\infty$, nothing can be said about the operator $T=\sum_j T_j$. This is the origin of many difficult and open problems. However,  if we assume that 
$$
  T_j:L^1(u)\longrightarrow L^{1, \infty}(u),  \qquad \forall u\in A_1,
  $$
with norm less than or equal to $\varphi(||u||_{A_1})||T_j||$, where $\varphi$ is a nondecreasing function and $A_1$ the Muckenhoupt class of weights, then we prove that, essentially,
$$
T:L^1(u) \longrightarrow L^{1, \infty}(u),  \qquad \forall u\in A_1. 
$$
We shall see that this is the case of many interesting problems in Harmonic Analysis.

\end{abstract}

\maketitle

\pagestyle{headings}\pagenumbering{arabic}\thispagestyle{plain}

\markboth{On weak-type $(1,\,1)$ for averaging type operators}{}

\section{Introduction}
Let $\{T_\theta\}_\theta$ be a family of operators indexed in a probability measure space such that 
\begin{equation}\label{unif1}
T_\theta:L^1(\mathbb R^n) \longrightarrow  L^{1, \infty}(\mathbb R^n)
\end{equation}
with norm less than or equal to a uniform constant $C$. What can we say about the  boundedness of the average operator
$$
T_A f(x)= \int T_\theta f(x) dP(\theta), \qquad x \in \mathbb R^n,
$$
whenever is well defined? The following trivial example shows that, at first sight, nothing of interest can be concluded: for $0 < \theta < 1$, set
$$
T_\theta f(x)= \frac{\int_0^1 f(y)dy}{|x-\theta|},  \qquad   x \in (0,1),
$$
so clearly $T_\theta$ satisfies \eqref{unif1}, but $$T_Af(x) = \int_0^1T_\theta f(x) d\theta \equiv \infty, \qquad \forall x \in (0,1).$$ 

However, things change completely, and this is one of the main goals of this paper,  if we assume that 
\begin{equation*}
T_\theta:L^1(u) \longrightarrow  L^{1, \infty}(u), \qquad \forall u \in A_1,
\end{equation*}
where $A_1$ is the class of Muckenhoupt weights defined as follows: we say that $u\in A_1$ if  $u$ is a nonnegative locally integrable function (called weight) so that there exists a positive constant $C$ such that 
$$
Mu(x)\le C u(x), \qquad \text{a.e. } x\in \mathbb R^n, 
$$
where $M$ is the Hardy-Littlewood maximal operator defined by 
$$
Mf(x)=\sup_{Q \ni x} \frac 1{|Q|} \int_Q |f(y)|\, dy,\qquad f \in L^1_{\text{loc}}(\mathbb R^n),
$$
with the supremum being taken over all cubes $Q \subseteq \mathbb{R}^n$ containing $x \in \mathbb{R}^n$. We denote by $\Vert u\Vert _{A_1}$ the least constant $C$ satisfying such inequality. Besides, it is wellknown (\cite{b:b,m:m}) that 
$$
M:L^1(u)\longrightarrow L^{1, \infty}(u)\quad\iff\quad u\in A_1, 
$$
with $||M||_{L^1(u)\rightarrow L^{1,\infty}(u)} \leq C||u||_{A_1}$.

\

Let us start with a very simple and motivating example. Let $m$ be a bounded variation function on $\mathbb R$ that is right-continuous and normalized by the condition $m(-\infty)=0$. Then,
$$
m(\xi)=\int_{-\infty}^\xi  dm(t) =\int_{\R}\chi_{(-\infty,\xi)}(t)\, dm(t) = \int_{\R}\chi_{(t,\infty)}(\xi)\, dm(t), \qquad \forall \xi \in \mathbb R,
$$
where $dm$ is the Lebesgue-Stieltjes measure associated with $m$ and it is a finite measure. Hence, if we consider the Fourier multiplier operator 
$$
T_mf(x)=\int_{\mathbb R} m(\xi) \hat f(\xi) e^{2\pi i x\xi}  d\xi, \qquad x \in \mathbb R, 
$$
for every Schwartz function $f$, where
$$
\hat f(\xi)=\int_\mathbb R f(x) e^{-2\pi i x\xi} dx, \qquad \xi \in \mathbb R,
$$
is the Fourier transform of the function $f$, a formal computation shows that 
$$
T_mf(x) = \int_\mathbb R  H_t f(x) dm(t), \qquad \forall x \in \mathbb R,
$$
where 
$$
H_t f(x)= T_{\chi_{(t, \infty)} } f(x) = \int_t^\infty \hat f(\xi) e^{2\pi i x\xi}  d\xi, \qquad x \in \mathbb R.
$$
Now, $H_t$ is essentially a Hilbert transform operator (recall that $Hf=T_{m}f$ with $m(\xi)=-i\sgn \xi$)   because 
$$
\chi_{(t, \infty)}(\xi)=\frac{\sgn (\xi-t)+1}2, \qquad \forall \xi \in \mathbb R. 
$$
Thus, since 
$$
H_t:L^p(\mathbb R) \longrightarrow L^p(\mathbb R), \qquad \forall p>1,
$$
we have, using the Minkowski's integral inequality and the density of the Schwartz functions on $L^p(\mathbb R)$, that every right-continuous bounded variation function is a Fourier multiplier on $L^p(\mathbb R)$ for every $p>1$. However, even though we also have
$$
H_t:L^1(\mathbb R) \longrightarrow L^{1, \infty}(\mathbb R), 
$$
we cannot deduce (at least not immediately) that the same boundedness holds for $T_m$ due to the lack of the Minkowski's integral inequality for the space $L^{1, \infty}(\mathbb R)$. 

The main theorem of this paper will show that since
$$
H_t:L^1(u) \longrightarrow L^{1, \infty}(u),  \qquad \varphi(||u||_{A_1}), \qquad \forall u \in A_1,
$$
with $\varphi$ being a nondecreasing function on $[1,\infty)$ and independent of $t\in \mathbb R$, then for every measurable set $E \subseteq \mathbb R^n$,
$$
||T_m\chi_E||_{L^{1, \infty}(u)} \leq C(m) \varphi(C_2||u||_{A_1})(1+\log ||u||_{A_1}) u(E), \quad \forall u \in A_1.
$$
The result will be proved using an extended version of the Rubio de Francia's extrapolation theorem which deals with the theory of Muckenhoupt weights (see Theorem~\ref{thrm:classical_rubio_de_Francia}). Other interesting applications will be given in Section~\ref{sec:examples}.

\

Let us now recall  (see \cite{b:b,m:m}) that for $p > 1$,
\begin{equation*}
M:L^p(v)\longrightarrow L^p(v) \qquad \Longleftrightarrow \qquad v \in A_p,
\end{equation*}
where this class of weights is defined by the condition 
$$
\Vert v\Vert _{A_p}=\sup_{Q\subseteq \mathbb R^n} \bigg(\frac{1}{|Q|} \int_Q v(x)\, dx\bigg)\bigg(\frac{1}{|Q|} \int_Q v(x)^{\frac1{1 - p}}\, dx\bigg)^{p-1}<\infty, 
$$
and, given a weight $v$, $L^p(v)$ is the Lebesgue space defined as the set of measurable functions $f$ such that 
$$
||f||_{L^p(v)} = \left(\int_{\mathbb R^n} |f(x)|^pv(x)\,dx\right)^\frac 1p < \infty.
$$ Indeed, (see \cite{fs:fs}) for every $p \geq 1$, 
$$
M:L^p(v)\longrightarrow L^{p, \infty}(v) \quad\iff\quad v\in A_p, 
$$
with $L^{p, \infty}(v)$ being the weak Lebesgue space defined as the set of measurable functions $f$ so that 
$$
\Vert f\Vert _{L^{p,\infty}(v)}= \sup_{y>0} y \lambda_{f}^v(y)^\frac 1p <\infty. 
$$
Here, $\lambda_f^v$ is the distribution function of $f$ with respect to $v$ defined by
% $$
% f^*_v(t)=\inf\Big\{y>0: \lambda_f^v(y)\le  t\Big\}, \qquad t > 0,
% $$
% and 
$$\lambda_f^v(y)=v\big(\big\{x\in\mathbb R^n: |f(x)|>y\big\}\big),\qquad y > 0.$$% is the distribution function of $f$ with respect to $v$. 
(Here we are using the standard notation $v(E)=\int_E v(x)\, dx$ for every measurable set $E \subseteq \mathbb R^n$. If $v=1$, we shall write $\lambda_f$ and  $|E|$. See \cite{bs:bs} for more details about this topic.) 

\

An important result  for our purpose concerning  $A_p$ weights is the extrapolation theorem of Rubio de Francia  \cite{rf:rf2,rf:rf} (see also \cite{cmpL:cmpL,d:d,gc:gc, gr:gr, g:g}) which, nowadays, can be formulated as follows:

\begin{theorem}[\cite{d:d}]\label{thrm:classical_rubio_de_Francia}
Let $(f, g)$ be a pair of measurable functions such that for some $1 \leq p_0 < \infty$,
$$
||g||_{L^{p_0}(v)} \le  \varphi(\Vert v\Vert _{A_{p_0}}) ||f||_{L^{p_0}(v)}, \qquad \forall v\in A_{p_0},
$$
with $\varphi$ being a nondecreasing function on $[1, \infty)$. Then, for every $1<p<\infty$,
$$
||g||_{L^p(v)} \le C_1\varphi\Big( C_2 \Vert v\Vert _{A_p}^{\max \big(1, \frac{p_0-1}{p-1}\big)}\Big)  ||f||_{L^p(v)}, \qquad \forall v \in A_p,
$$
with $C_1$ and $C_2$ being two positive constants independent of all parameters involved.
\end{theorem}

We have to emphasize here that although $p_0$ can be $1$, it is not possible, in general, to extrapolate till the endpoint $p=1$ (take just $T=M\circ M$ or see, for instance, \cite{p:p} where a counterexample is given in the case of commutators). However, in  the recent papers  \cite{cgs:cgs, cs:cs}, a Rubio de Francia extrapolation theory for operators satisfying a weighted restricted weak-type boundedness for the  class of weights $\widehat A_p$ (slightly bigger than the class $A_p$) has been developed. The main advantage of this new class of weights  is that  allows to obtain boundedness estimates at the endpoint $p=1$.  

\begin{definition}  We define 
$$
\widehat A_p =\Big\{ v\in L^1_{\loc}(\mathbb R^n): \exists \,  h \in L^1_{\loc}(\mathbb R^n) \mbox{ and } \exists \,  u\in A_1 \text{ with }  v = (Mh)^{1-p} u\Big\}, 
$$
endowed with the norm
$$
\Vert v\Vert _{\widehat A_{p}} =\inf \left\{\Vert u\Vert_{A_1}^{\frac 1p} : v = (Mh)^{1 - p}u\right\}.
$$
Clearly, $\widehat A_1 = A_1$, while for $1 < p < \infty$, $A_p \subsetneq \widehat A_p$. 
\end{definition}

It holds that (see \cite{cgs:cgs,chk:chk, kt:kt}) for every $1\le p<\infty$ and every $v\in \widehat A_p$,
$$
M:L^{p, 1}(v) \longrightarrow L^{p, \infty}(v), \qquad \Vert M\Vert_{L^{p, 1}(v) \longrightarrow L^{p, \infty}(v)} \leq C\Vert v\Vert_{\widehat A_p}, 
$$
where the Lorentz space $L^{p,1}(v)$ is defined as   the set of  measurable functions $f$ such that
$$
\Vert f\Vert _{L^{p,1}(v)} = p\int_0^\infty  \lambda_f^v(y)^\frac1p\,dy  <\infty. 
$$

Then, the restricted weak-type Rubio de Francia extrapolation result proved in \cite{cgs:cgs} can be stated as follows: 

\begin{theorem}[\cite{cgs:cgs}]\label{thrm:first_extrap_result}
Let $1<p_0<\infty$ and let $T$ be an operator such that 
$$
T:L^{p_0, 1}(v)\longrightarrow L^{p_0, \infty}(v), \qquad \varphi(\Vert v\Vert _{\widehat A_{p_0}}), \qquad \forall v\in \widehat A_{p_0},
$$
where $\varphi$ is a positive nondecreasing function on $[1,\infty)$. Then, $T$ is of weighted restricted weak-type $(1,\,1)$ for every weight in $A_1$; that is, for any measurable set $E\subseteq\mathbb R^n$, there exists a constant $C>0$ independent of $E$ such that
\begin{equation}\label{eq:first_extrap_result}
\Vert T\chi_E\Vert _{L^{1, \infty}(u)} \le C   \Vert u\Vert _{ A_{1}}^{1-\frac 1{p_0} }\varphi\left( \Vert u\Vert _{ A_{1}}^{\frac 1{p_0}}\right) u(E), \qquad \forall u \in A_1.
\end{equation}
\end{theorem}

For simplicity, whenever an operator $T$ satisfies that for every measurable set $E$, 
$$
\Vert T\chi_E\Vert _{L^{1, \infty}(u)} \le C_u  u(E), 
$$
we shall denote it by 
$$
T:L^1_{\mathcal R}(u) \longrightarrow L^{1, \infty}(u), \qquad C_u.
$$
\begin{remark} The complete result that $T$ is of weighted weak-type $(1,\,1)$ (i.e., that the estimate in \eqref{eq:first_extrap_result} holds for every $f\in L^1(u)$) is, in general, false (see \cite{cgs:cgs}). However, under certain mild condition in the operator $T$ (see Section \ref{epsilondelta}) the weighted weak-type $(1,\,1)$ boundedness can be proved. 
\end{remark}
\begin{remark}
We should emphasize here that our operators do not need to be sublinear. However, if $T$ is sublinear, it was proved in \cite{s:s} that 
$$
T:L^1_{\mathcal R}(u) \longrightarrow L^{1, \infty}(u)
$$  
is equivalent to have the boundedness on the space  
$$
B^*(u)=\left\{f: \int_0^\infty \lambda_f^u(t) \left(1+\log \frac{||f||_1}{\lambda_f^u(t)}\right)dt <\infty\right\},
$$
which can be endowed with a quasi-norm. 
\end{remark}

\medskip
Our main goal will be consequence of the fact that the converse of Theorem \ref{thrm:first_extrap_result} is also true, and hence $$
T:L^1_{\mathcal R}(u)\longrightarrow L^{1, \infty}(u), \, \forall u\in   A_1 \, \iff \, T:L^{p_0, 1}(v)\longrightarrow L^{p_0, \infty}(v), \, \forall v\in \widehat A_{p_0}.
$$
Indeed, if $p'$ is the conjugate exponent of $p > 1$ (that is, $\frac 1p + \frac 1{p'} = 1$) our main theorem reads as follows:

\begin{theorem}\label{thrm:principal_extrap_result}
Let $(f, g)$ be a pair of measurable functions such that
$$
||g||_{L^{1,\infty}(u)} \le  \varphi(\Vert u\Vert _{A_1}) ||f||_{L^{1}(u)}, \qquad \forall u\in A_1,
$$
with $\varphi$ being a nondecreasing function on $[1, \infty)$. Then, for every $1<p<\infty$,
$$
||g||_{L^{p,\infty}(v)} \le \Phi(\lVert v\rVert _{\widehat A_p})  ||f||_{L^{p,1}(v)}, \qquad \forall v \in \widehat A_p,
$$
where 
\begin{equation*}%\label{eq:phi_expression2}
    \Phi(r) = C_1 \varphi(C_2 r^p) r^{p-1}(1 + \log r) ^{\frac 2{p'}}, \qquad r \geq 1,
\end{equation*} with $C_1$ and $C_2$ being two positive constants independent of all parameters involved. 
\end{theorem}

As a consequence we obtain the following corollary:

\begin{corollary}\label{cor:average_cor_Tj_cj}
Let $c=(c_j)_j\in \ell^1$ and let $\{T_j\}_j$ be such that
$$
T_j:L^1(u) \longrightarrow L^{1, \infty}(u), \qquad \varphi(||u||_{A_1}),\qquad \forall u \in A_1,
$$
where $\varphi$ is a positive nondecreasing function on $[1,\infty)$. Then,   for every $u\in A_1$, 
$$
\sum_j c_j T_j:L^1_{\mathcal R}(u) \longrightarrow L^{1, \infty}(u), \qquad   C_1||c||_{\ell^1}\varphi(C_2||u||_{A_1})(1+\log ||u||_{A_1}).
$$
 
\end{corollary}

As usual, we shall use the symbol $A\lesssim B$ to indicate that there exists a universal positive constant $C$,  independent of all important parameters, such that $A\le C B$. When $A\lesssim B$ and $B\lesssim A$, we will write $A\approx B$.

\medskip

The paper is organized as follows. In Section \ref{sec:def_technic}, we will see some previous notions, the necessary definitions and some technical results which shall be used later on. Indeed, there we will prove Lemma \ref{difi} which will be essential in the proof of the main result given in  Section \ref{sec:main}. Further, Section \ref{sec:examples} contains our main examples and applications. Finally, we also include a last section related with similar results in the context of limited extrapolation. 

\section{Preliminary notions and some technical results}\label{sec:def_technic}

\subsection{$A_1$ weights}
 
Let us start by recalling some wellknown facts of the class $A_1$: 

\noindent
i) (\cite[Theorem 7.7]{d:d2}) A weight $u$ belongs to $A_1$ if and only if there exists $h\in L^1_{\text{loc}}(\mathbb R^n)$ and $K$ such that $K, K^{-1}\in L^\infty(\mathbb R^n)$ satisfying that, for some $0<\mu<1$, 
$$
u(x)= K(x) (Mh(x))^\mu, \qquad \text{a.e. } x \in \mathbb R^n,
$$
where $L^\infty(\mathbb R^n)$ consists of all measurable functions $f$ such that $$||f||_\infty := ||f||_{L^\infty(\mathbb R^n)} = \esssup{f} < \infty. $$

\noindent
ii) (\cite[Lemma 2.12]{cs:cs}) For  every $h\in L^1_{\text{loc}}(\mathbb R^n)$, every $u\in  A_1$ and $0<\mu<1$, then $(Mh)^{ \mu }  u ^{1 - \mu} \in A_1$ with 
\begin{equation}\label{eq:MfumuA1}
    \bigg\Vert (Mh)^{\mu}  u^{1 - \mu} \bigg\Vert_{A_1} \lesssim \frac {\Vert u\Vert_{A_1}}{1 - \mu}.
\end{equation}

\noindent iii) (\cite[Lemma 5.1]{p:p2}) If $t = 1 + \frac 1{2^{n + 1}\lVert u \rVert_{A_1}}$, then \begin{equation}\label{eq:u^t_A1}
    u^t\in A_1 \quad \text{ and } \quad ||u^t||_{A_1} \lesssim ||u||_{A_1}.
\end{equation}

\subsection{$(\varepsilon,\delta)$-atomic operators}\label{epsilondelta}

As mentioned above, in general, the following implication does not hold for every $u\in A_1$:
$$
T:L^1_\mathcal R(u) \longrightarrow L^{1, \infty}(u) \qquad \Longrightarrow \qquad  T:L^1(u) \longrightarrow L^{1, \infty}(u), 
$$

\noindent even if $T$ is a sublinear operator. However, it was proved in \cite[Theorem~3.5]{cgs:cgs} that for a quite big class of operators the above implication is true. 

\begin{definition}  Given $\delta>0$, a function $a\in L^1(\mathbb{ R}^n)$ is called a \emph{$\delta$-atom} if it satisfies the following
properties:

\begin{enumerate}[\rm (i)]
\item $\int_{\mathbb{ R}^n} a(x)\, dx=0$, and 

\item
there exists a cube $Q \subseteq \mathbb R^n$ such that $|Q|\le \delta$ and $\mbox{\rm supp }a\subseteq Q$.
\end{enumerate}
\end{definition}

\begin{definition}%\label{def:eps_del_approx}
(a)  A sublinear operator $T$ is called \emph{$(\varepsilon,  \delta)$-atomic} if, for every
$\varepsilon>0$, there exists $\delta>0$ satisfying that
\begin{equation*}%\label{atomcon}
\Vert {Ta}\Vert_ {L^1(\mathbb R^n)+L^\infty(\mathbb R^n)}\le \varepsilon \Vert a\Vert_1,
\end{equation*}
for every $\delta$-atom  $a$. 

\noindent
(b) A sublinear operator $T$ is said to be \emph{$(\varepsilon,  \delta )$-atomic approximable}  if there exists a sequence $\{T_j\}_j$ of $(\varepsilon,  \delta)$-atomic operators such that, for every measurable set $E \subseteq \mathbb R^n$, then $|T_j\chi_E|\le |T\chi_E|$ and, for every
$f\in L^1(\mathbb R^n)$ such that
$\Vert f\Vert_\infty\le 1$,
$$
|Tf(x)|\le \lim_j\inf |T_jf(x)|, \qquad \text{a.e. } x \in \mathbb R^n.
$$

\end{definition}

\noindent
{\bf Examples: } In \cite{c:c}, the author showed that for sublinear operators, the property of being $(\varepsilon, \delta)$-atomic is not a strong one. For instance, if $$Tf(x) = K*f(x) = \int_{\mathbb R^n} K(y - x)f(y)\,dy, \qquad x \in \mathbb R^n,$$ with $K \in L^p(\mathbb R^n)$ for some $1 \leq p < \infty$, then $T$ is $(\varepsilon,\delta)$-atomic.  Further, if 
$$
T^*f(x)=\sup_{ j\in\mathbb{ N}} \bigg|\int_{\mathbb{ R}^n} K_j(x,y) f(y)\, dy\bigg|,  \qquad x \in \mathbb R^n,
$$  
with 
\begin{equation*}
\lim_{y\to x }\Vert {K_j(\,\cdot\,, y)-K_j(\,\cdot\,, x ) }\Vert_{L^1(\mathbb R^n)+L^\infty(\mathbb R^n)}=0,
\end{equation*}
then $T^*$  is $(\varepsilon, \delta)$-atomic approximable (for example, standard maximal Calder\'on -Zygmund operators are of this type).  In general, $$T^*f(x)=\sup_j |T_j f(x)|, \qquad x \in \mathbb R^n,$$ where $\{T_j\}_j$ is a sequence of  $(\varepsilon,
\delta)$-atomic, is 
$(\varepsilon, \delta)$-atomic approximable and the same holds for  $$Tf(x)=\bigg(\sum_j |T_j f(x)|^q\bigg)^{\frac1q}, \qquad x \in \mathbb R^n,$$ with $q \in [1,\infty)$ and $$Tf(x)=\sum_j T_j f(x), \qquad x \in \mathbb R^n.$$ (See \cite{c:c,cgs:cgs} for more examples.)

\begin{theorem}[\cite{cgs:cgs}]\label{thrm:from_restricted_to_all}   Let $T$ be a sublinear operator $(\varepsilon, \delta)$-atomic approximable. Then, given $u \in A_1$,
$$
T:L^1_{\mathcal R} (u) \longrightarrow L^{1, \infty}(u), \quad C_u \quad\implies\quad T:L^1 (u) \longrightarrow L^{1, \infty}(u), \quad 2^nC_u \Vert u\Vert _{A_1}.
$$
\end{theorem}

\bigskip

\subsection{A Sawyer-type inequality}%\label{subsec:lemma_sawyer}
Here we will study one of the often-called Sawyer-type inequalities for weights belonging in the restricted class of weights $\hat A_p$. First, to do so, we need the following result.

\begin{lemma}\label{lemma:difi_prev}
Let $1<p<\infty$ and $v\in \widehat A_p$. Take $\frac 1{p'} < \theta \leq 1$ and set $u_0 = (Mh)^{\frac{(p - 1)(1 - \theta)}\theta}$. Then, \begin{equation}\label{eq:lemma_difi_prev}
    M_{u_0}:L^{\frac{\theta p'}{\theta p' - 1}, 1}(v) \longrightarrow L^{\frac{\theta p'}{\theta p' - 1},\infty}(v),
\end{equation} with constant less than or equal to  $$\frac{\theta^2 p' C_{n,p}}{1 - p(1 - \theta)} \norm{v}_{\hat A_p}^{\frac{2(\theta p' - 1)}{\theta(p'-1)}},$$ and where $$M_{u_0}f(x) = \sup_{Q \ni x} \frac 1{u_0(Q)} \int_Q |f(y)| u_0(y) \,dy, \qquad x \in \mathbb R^n.$$  
\end{lemma}

\begin{proof}
Observe that since $v \in \hat A_p$, then $v$ is a doubling weight with constant $\Delta_v \leq C_1\norm{v}_{\hat A_p}^p$. Therefore, according to \cite[Lemma 2.2 (i)]{cs:cs}, \eqref{eq:lemma_difi_prev} is bounded with constant less than or equal to $$C_1\norm{v}_{\hat A_p}^{\frac{\theta p' - 1}{\theta (p'- 1)}}\theta p'\left[ \sup_{E\subseteq Q} \frac{u_0(E)}{u_0(Q)}\left(\frac{v(Q)}{v(E)}\right)^\frac{\theta p' - 1}{\theta p'} \right],$$ where the supremum is taken over all cubes $Q$ and all measurable sets $E \subseteq Q$. 

Now, given a cube $Q$ and a measurable set $E \subseteq Q$, 
\begin{align*}
   \left(\frac{v(Q)}{v(E)}\right)^\frac{\theta p' - 1}{\theta p'} &= \left(\frac{|Q|}{|E|}\right)^{\frac{\theta p' - 1}{\theta(p' - 1)}} \left[\left(\frac{|E|}{|Q|}\right)^p\frac{v(Q)}{v(E)}\right]^\frac{\theta p' - 1}{\theta p'} \\ &\leq C_2 \norm{v}_{\hat A_p}^{\frac{\theta p' - 1}{\theta(p' - 1)}} \left(\frac{|Q|}{|E|}\right)^{\frac{\theta p' - 1}{\theta(p' - 1)}}
\end{align*}
and, as well, due to \cite[Lemma 2.5]{cs:cs}, \begin{align*}
    \sup_{E\subseteq Q }\frac{u_0(E)}{u_0(Q)}\left(\frac{|Q|}{|E|}\right)^{\frac{\theta p' - 1}{\theta(p' - 1)}} \leq \frac{ \theta C_3}{1 - p(1 - \theta)},
\end{align*}
which yields the desired result.
\end{proof}

The following lemma was proved for the case $\mu=1$ in \cite[Lemma 2.6]{cs:cs}, and the extension to other $\mu$'s has been fundamental for our purposes.  
\begin{lemma}\label{difi} 

\noindent Let $1<p<\infty$ and let $v= (Mh)^{1-p} u\in \widehat A_p$.  Take $\theta$ and $\mu$ so that $\frac 1{p'}<\theta<\mu\le 1$ and set $v_\theta= (Mh)^{1-p} u^{\theta}$. Then, 
\begin{equation*}
\Big\Vert \frac{M_\mu(\chi_E v_\theta)}{v_\theta}\Big\Vert_{L^{p', \infty}(v)} \lesssim  C_{p, \theta, \mu}(u) v(E)^\frac 1{p'}, \qquad \forall E\subseteq \mathbb R^n, 
\end{equation*}
where $M_\mu f := M(|f|^{1/\mu})^\mu$ and
 \begin{equation}\label{C_pthetamu(u)} 
C_{p, \theta, \mu}(u) = \bigg(\frac {p^2}{(p-1)^2(\mu-\theta) (\theta-\frac 1{p'})^2}\bigg)^{\theta}\Vert u\Vert_{ A_1}^{2\theta - \frac 2{p'}}. 
 \end{equation}
 
\end{lemma}

\begin{proof} Observe that in virtue of the Kolmogorov's inequality \cite{gM:gM} with $1<r'=\frac 1\theta<p'$, it is enough to prove that
 \begin{align*}
 &\sup_{F\subseteq\mathbb R^n} \frac 1{v(F)^{\frac 1{r'}-\frac1{p'}}}\bigg(\int_F  (Mh(x))^{(p-1)(r'-1)} \big(M_\mu(\chi_E (Mh)^{1-p}u^\theta)(x)\big)^{r'} dx\bigg)^{\frac 1{r'}}
\\
\lesssim &  \,C_{p, \theta, \mu}(u) v(E)^{\frac 1{p'}}.
 \end{align*}

Then, using the Fefferman-Stein's inequality \cite{fs:fs}, since $\mu r'>1$, we obtain that 
\begin{align*}
&\int_F  (Mh(x))^{(p-1)(r'-1)} \big(M_\mu(\chi_E (Mh)^{1-p}u^\theta)(x)\big)^{r'} dx \\
\lesssim & \frac {\mu r'}{\mu r'-1} \int_E (Mh(x))^{(1-p)r'} 
M(\chi_F (Mh)^{(p-1)(r'-1)})(x)  u(x)dx.
\end{align*}

Now, since $u_0 = (Mh)^{(p-1)(r'-1)} \in A_1$, we have that, for every $x \in E$ and every cube $Q \ni x$ in $\mathbb R^n$, \begin{equation}\label{lemmamutheta_eq1}
\begin{split}
    \frac 1{|Q|} \int_Q \chi_F u_0(y)\,dy &\leq \frac{u_0(Q)}{|Q|} M_{u_0}(\chi_F)(x) \leq ||u_0||_{A_1}u_0(x)M_{u_0}(\chi_F)(x) \\ &\lesssim \frac 1{1 - (p-1)(r' - 1)}u_0(x)M_{u_0}(\chi_F)(x),
\end{split}
\end{equation} where in the last estimate we have used \eqref{eq:MfumuA1}. Hence, taking the supremum over all cubes $Q \in \mathbb R^n$ such that $Q \ni x$ in \eqref{lemmamutheta_eq1}, with $x \in E$, we deduce that \begin{align*}
    &\int_E (Mh(x))^{(1-p)r'} 
M(\chi_F (Mh)^{(p-1)(r'-1)})(x)  u(x)dx \\ \lesssim & \frac 1{1 - (p-1)(r' - 1)} \int_E M_{u_0}(\chi_F)(x)v(x)\,dx.
\end{align*}

Therefore, since $r' = \frac 1\theta$, the inequality we want to prove will hold if we see that \begin{align*}
    \sup_{E \subseteq \mathbb R^n} &\frac 1{v(E)^{\frac 1{p'}}}\left(\int_E M_{u_0}(\chi_F)(x)v(x)\,dx\right)^\theta \\ &\lesssim \left( \frac{(\mu - \theta)(1 - p(1 - \theta))}{\mu \theta} \right)^\theta C_{p,\theta, \mu}(u) v(F)^{\theta - \frac 1{p'}}
\end{align*} or equivalently, \begin{equation}\label{lemmamutheta_eq2}
\begin{split}
    \sup_{E \subseteq \mathbb R^n} &\frac 1{v(E)^{1 - \left(1 - \frac 1{\theta p'}\right)}}\int_E M_{u_0}(\chi_F)(x)v(x)\,dx \\ &\lesssim \left( \frac{(\mu - \theta)(1 - p(1 - \theta))}{\mu \theta} \right) C_{p,\theta, \mu}(u)^\frac 1\theta v(F)^{1 - \frac 1{\theta p'}}.
\end{split}
\end{equation}

Finally, using again the Kolmogorov's inequality in \eqref{lemmamutheta_eq2}, it is enough to prove that \begin{equation*}%\label{Mu0_boundedness}
    M_{u_0}:L^{\frac{\theta p'}{\theta p' - 1}, 1}(v) \longrightarrow L^{\frac{\theta p'}{\theta p'- 1}, \infty}(v)
\end{equation*} with constant less than or equal to $$\frac {c_{n,p}}{\theta p'}\left( \frac{(\mu - \theta)(1 - p(1 - \theta))}{\mu \theta} \right) C_{p,\theta, \mu}(u)^\frac 1\theta.$$ 

According to Lemma~\ref{lemma:difi_prev}, this will happen if $$C_{p,\theta, \mu}(u) \gtrsim \left(\frac{p^2}{(p - 1)^2(\mu - \theta)(1 - p(1-\theta))^2}\right)^\theta \norm{u}_{A_1}^{\frac{2(\theta p' - 1)}{ p'}},$$ from which the desired result follows by taking $C_{p,\theta,\mu}(u)$ as in \eqref{C_pthetamu(u)}.

\end{proof}

\section{Proof of the main result}\label{sec:main}

We are now ready to prove our main result:

\begin{proof}[Proof of  Theorem \ref{thrm:principal_extrap_result}] Let $h \in L^1_{\text{loc}}(\mathbb R^n)$ and $u \in A_1$ so that $v = (Mh)^{1-p} u\in \widehat A_p$. Further, let us take
$$ \frac 1{p'}<\theta <1, \qquad \mu:= 1- \frac{1-\theta}t \qquad \text{ and } \qquad v_\theta:= (Mh)^{1-p} u^\theta,
$$
where $t = 1 + \frac 1{2^{n + 1}\lVert u \rVert_{A_1}}$ satisfies $u^t\in A_1$ and $||u^t||_{A_1} \lesssim ||u||_{A_1}$ (see \eqref{eq:u^t_A1}). Then, $\theta < \mu < 1$ and, by \eqref{eq:MfumuA1}, for every measurable set $F \subseteq \mathbb R^n$, $$
u_0 = M_\mu( \chi_F v_\theta) u^{1-\theta} =M( \chi_F v_\theta^{1/\mu})^\mu (u^t)^{1-\mu} \in A_1, \qquad ||u_0||_{A_1}\leq \frac{C||u||_{A_1}}{1-\mu}.
$$

Let $y > 0$ and set $F=\{x : |g(x)|>y\}$ so that $v(F) = \lambda^v_g(y)$. We can assume, without lost of generality, that $v(F)<\infty$, since on the contrary we can take $g_N=g\chi_{B(0,N)}$ and let $N$ go to infinity at the end of our estimate. 

By hypothesis we obtain that
\begin{eqnarray*}
y\lambda_g^v(y) &=& y\int_{\{x \,:\,|g(x)|>y\}} v(x)\,dx \le y \int_F M_\mu( \chi_F v_\theta )(x) u(x)^{1-\theta}  dx
\\
&\leq&  \varphi \left(\frac{C||u||_{A_1}}{1-\mu}\right)\int_{\mathbb R^n} |f(x)| M_\mu( \chi_F v_\theta )(x) u(x)^{1-\theta}  dx
\\
&=&  \varphi \left(\frac{Ct||u||_{A_1}}{1-\theta}\right) \int_{\mathbb R^n} |f(x)| \frac{M_\mu( \chi_F v_\theta )(x)}{v_\theta(x)}  \, v(x)  \,  dx
\\
&\le &   \varphi \left(\frac{Ct||u||_{A_1}}{1-\theta}\right) \left \Vert \frac{M_\mu( \chi_F v_\theta) }{v_\theta}  \right \Vert_{L^{p',\infty}(v)}||f||_{L^{p,1}(v)},
\end{eqnarray*}

\noindent where in the last estimate we have used the H\"older's inequality for Lorentz spaces with respect to the measure $v(x)\,dx$. 

Now, by virtue of Lemma \ref{difi},

\begin{eqnarray*}
\left \Vert \frac{M_\mu( \chi_F v_\theta) }{v_\theta}  \right \Vert_{L^{p',\infty}(v)} &\lesssim &  C_{p, \theta, \mu}(u) v(F)^\frac 1{p'} = C_{p, \theta, \mu}(u) \lambda^v_g(y)^\frac 1{p'},
\end{eqnarray*}

\noindent so taking the supremum over all $y > 0$, in particular, we obtain that $$ \lVert g \rVert_{L^{p,\infty}(v)} \lesssim C_{p, \theta, \mu}(u) \varphi\left(\frac{Ct||u||_{A_1}}{1-\theta}\right)  ||f||_{L^{p,1}(v)}.$$

Finally, concerning about the constant $C_{p,\theta, \mu}(u)$, we observe that
\begin{align*}
    C_{p, \theta, \mu}(u) & = \bigg(\frac {p^2}{(p-1)^2(\mu-\theta) (\theta-\frac 1{p'})^2}\bigg)^{\theta}\Vert u\Vert_{ A_1}^{2\theta - \frac 2{p'}} \\ & \approx\left( \frac {p^2}{(p-1)^2(1 - \theta)(\theta - \frac 1{p'})^2} \right)^\theta \lVert u \rVert_{A_1}^{3\theta - \frac2{p'}}. 
\end{align*}

\noindent Therefore, letting $$\theta = \frac 1{p'}\left(1 + \frac 1{(p + 1)R}\right), \qquad 1 \leq R < \infty,$$  then \begin{align*}
    C_{p, \theta, \mu}(u) &\lesssim \left( \frac {p^5(p+1)^3R^2}{(p-1)^4} \right)^{\frac 1{p'}\left(1 + \frac 1{(p + 1)R}\right)} \norm{u}_{A_1}^{\frac 1{p'}} \norm{u}_{A_1}^{\frac 3{Rp'(p+1)}} \lesssim R^{\frac 2{p'}} \norm{u}_{A_1}^{\frac 1{p'}} \norm{u}_{A_1}^{\frac 3R}.  
\end{align*}

\noindent Furthermore, with the same choice of $\theta$, $$\varphi \left(\frac{Ct ||u||_{A_1}}{1-\theta}\right)  \leq \varphi \left(\tilde C ||u||_{A_1}\right). $$

Thus, the result  follows by setting $R = 1 + \log \lVert u \rVert_{A_1}$ and then taking the infimum on $\lVert u \rVert_{A_1}$ over all  possible representations of $v \in \widehat A_p$. 
\end{proof}

\section{Examples and applications to average operators, multipliers and integral operators}\label{sec:examples}
  
\subsection{Examples}
  
There are many operators in harmonic analysis for which the weak-type $(1,\,1)$ boundedness for every weight in $A_1$ has been proved \cite{cgs:cgs,hp:hp,kw:kw,lop:lop,lprr:lprr,v:v}.
  
As a consequence of the classical Rubio de Francia extrapolation theory (see Theorem~\ref{thrm:classical_rubio_de_Francia}) it is known that they are also bounded on $L^p(v)$ for every $v\in A_p$; but, in general, the restricted weak-type 
$$
T:L^{p,1}(v) \longrightarrow L^{p,\infty}(v), \qquad \forall v\in\widehat A_p,
$$
has been unknown up to now for many examples.  This is the case, for instance, of the Bochner-Riesz operator at the critical index $B_{\frac{n-1}2}$, introduced by S. Bochner in \cite{b:b2} and defined as follows (see \cite{cd:cd} for some partial results in this context): 
let $a_{+} = \max\{a,0\}$ denote the positive part of $a \in \mathbb R$ and
given $\lambda > 0$, the Bochner-Riesz operator $B_{\lambda}$ on $\mathbb R^n$ is defined by 
$$
\widehat{B_{\lambda}f}(\xi) = \left( 1 - |\xi|^2  \right)^{\lambda}_{+}\hat{f}(\xi), \qquad \xi \in \mathbb R^n.
$$

\begin{proposition}[\cite {lprr:lprr, v:v}]%\label{prop:bochner_critical_index}
For every $n > 1$, 
$$
B_{\frac{n-1}2}:L^1(u)\longrightarrow L^{1,\infty}(u),\qquad C||u||^2_{A_1} \log(||u||_{A_1} + 1), \qquad \forall u \in A_1.
$$ 
\end{proposition}

Thereby, in virtue of Theorem~\ref{thrm:principal_extrap_result}, we completely answer the open question formulated in \cite{cd:cd} about the restricted weak-type boundedness of $B_{\frac{n-1}2}$. 

  \begin{corollary}
 For every $n > 1$ and every $p > 1$,
$$
B_{\frac{n-1}2}:L^{p,1}(v) \longrightarrow L^{p,\infty}(v), \qquad C||v||^{3p - 1}_{\hat{A}_p}(1 + \log||v||_{\hat{A}_p})^{1 + \frac 2{p'}}, \qquad \forall v\in\widehat A_p.
$$ 
\end{corollary}

Same estimates can be obtained for a large list of operators such as those appearing in \cite{bc:bc,cd:cd,cgs:cgs}: rough operators, H\"ormander multipliers, radial Fourier multipliers, square functions, etc.

\medskip

\subsection{\bf Average operators}

\medskip

\begin{corollary}\label{cor:operators_prob_meas} Assume that  $\{T_\theta\}_\theta$ is a family of operators indexed in a probability measure space such that the average operator 
$$
T_A f(x)= \int T_\theta f(x) dP(\theta), \qquad x \in \mathbb R^n,
$$
is well defined and that
\begin{equation}\label{eq:ttheta_A1_bound}
T_\theta:L^1(u)\longrightarrow L^{1, \infty}(u), \qquad \varphi(||u||_{A_1}), \qquad \forall u \in A_1,
\end{equation}
where $\varphi$ is a positive nondecreasing function on $[1,\infty)$. Then, 
\begin{equation}\label{rest}
T_A:L^1_{\mathcal R}(u) \longrightarrow L^{1, \infty}(u),\quad C_1\varphi(C_2||u||_{A_1})(1+\log ||u||_{A_1}),\quad \forall u \in A_1.
\end{equation}
Moreover, if $T_A$ is a sublinear $(\varepsilon, \delta)$-atomic approximable operator, then 
\begin{equation}\label{eq:TA_bounded_Lorentz_Sp_p=1}
T_A:L^1(u) \longrightarrow L^{1, \infty}(u),\qquad \tilde C_1\varphi(C_2||u||_{A_1})||u||_{A_1}(1+\log ||u||_{A_1}).
\end{equation}
\end{corollary}

\begin{proof} Set $1 < p < \infty$. Using Theorem \ref{thrm:principal_extrap_result},  we have that \eqref{eq:ttheta_A1_bound}  implies
$$
T_\theta:L^{p,1}(v) \longrightarrow  L^{p, \infty}(v),\qquad \Phi(||v||_{\widehat A_p}), \qquad \forall v\in\widehat A_p.
$$
Now,   $L^{p,\infty}(v)$ is a Banach function space since there exists a norm $||\cdot||_{(p,\infty,v)}$ so that
$$
||f||_{L^{p,\infty}(v)}\le ||f||_{(p,\infty,v)} \le \frac p{p-1} ||f||_{L^{p,\infty}(v)}. 
$$
Hence, by the Minkowski's integral inequality, 
$T_A$ satisfies that for every $p > 1$,
\begin{equation*}%\label{eq:TA_bounded_Lorentz_Sp}
    T_A:L^{p,1}(v) \longrightarrow  L^{p, \infty}(v), \qquad \frac p{p-1}\Phi(||v||_{\widehat A_p}), \qquad \forall v\in\widehat A_p.
\end{equation*}
Therefore, using Theorem~\ref{thrm:first_extrap_result} the desired result \eqref{rest}  follows by taking the infimum in $p>1$. Finally,  \eqref{eq:TA_bounded_Lorentz_Sp_p=1} is just a consequence of Theorem~\ref{thrm:from_restricted_to_all}.
\end{proof}

In particular, 
the next result stated in the introduction follows:

\begin{proof}[Proof of Corollary~\ref{cor:average_cor_Tj_cj}]
This result is just a direct consequence of Corollary~\ref{cor:operators_prob_meas} since $\left\{\frac{c_j}{||c||_{\ell^1}}T_j\right\}_j$ is a family of operators indexed in the counting probability measure.
\end{proof}

\medskip

\noindent
{\bf (I) Fourier multipliers}

\medskip

\noindent Our next application is in the context of restriction multipliers from $\mathbb R^{n+k}$ to $\mathbb R^n$. First, let us recall that a  bounded function $m$ defined on $\mathbb R^n$ is said to be \emph{normalized} if \begin{equation}\label{eq:normalized_def_m}
    \lim_j \widehat {\psi_j} * m(x) = m(x), \qquad \forall x \in \mathbb R^n,
\end{equation} where for each $j$, $\psi_j(x)= \psi(x/j)$, and $\psi \in \mathcal C^\infty_c(\mathbb R^n)$ (i.e., $\psi$ is an infinitely differentiable function with compact support), $\hat \psi \geq 0$ and $||\hat \psi||_1 = 1$.

It is easy to see that then, for every Lebesgue point $x$ of $m$, \eqref{eq:normalized_def_m} holds. In particular, every continuous and bounded function is normalized.

\begin{proposition}%\label{prop:restriction_multiplier}
Let  $k \geq 1$ and assume that a normalized bounded function $m$ defined in $\mathbb R^{n+k}$ satisfies that $$T_m:L^1(u) \longrightarrow L^{1,\infty}(u), \qquad  \varphi(||u||_{A_1}), \qquad \forall u \in A_1(\mathbb R^{n + k}),
$$  
where $\varphi$ is a positive nondecreasing function on $[1,\infty)$. Let $\phi\in L^1(\mathbb R^k)$ and define
$$
m_\phi(x)= \int_{\mathbb R^k} m(x, y) \phi(y)\, dy, \qquad x \in \mathbb R^n.
$$ 
Then, for every $v \in A_1(\mathbb R^{n})$,
$$
T_{m_\phi}:L^1_{\mathcal R}(v)\longrightarrow L^{1, \infty}(v), \qquad C_1\varphi(C_2||v||_{A_1})||v||_{A_1} (1+\log ||v||_{A_1}). 
$$
\end{proposition}

\begin{proof}

\noindent Take $v \in A_1(\mathbb R^n)$ and define $u = v\otimes \chi_{\mathbb R^k}$, so that $$\begin{array}{ccl}
     u \,:\, & \mathbb R^n \times \mathbb R^k & \longrightarrow \mathbb R,  \\
     & (x, \,y) & \longmapsto u(x,y) = v(x),
\end{array}$$ satisfies $u \in A_1(\mathbb R^{n + k})$ with $||u||_{A_1} \leq ||v||_{A_1}$. Then, $T_m:L^1(u) \longrightarrow L^{1,\infty}(u)$ and, by \cite[Theorem 4.4]{cs:cs2} (where here is used that $m$ is normalized), $$T_{m(\cdot, y)}:L^1(v) \longrightarrow L^{1,\infty}(v), \qquad \forall y \in \mathbb R^k,$$ with 
\begin{align*}
    \sup_{y \in \mathbb R^k} ||T_{m(\cdot, y)}||_{L^1(v) \rightarrow L^{1,\infty}(v)} &\lesssim ||u||_{A_1} ||T_m||_{L^1(u) \rightarrow L^{1,\infty}(u)} \leq ||v||_{A_1}\varphi(||v||_{A_1}).
\end{align*}

Now, take $f \in \mathcal C^\infty_c(\mathbb R^n)$. Then, for every $y \in \mathbb R^k$ we have that $m(\cdot, y)\hat f \in L^1(\mathbb R^n)$ and, as well, $m_\phi\hat f \in L^1(\mathbb R^n)$, so that, by the
properties of the Fourier transform,
\begin{align*}
    T_{m(\cdot, y)}f(x) &= \left(m(\cdot, y) \hat f\right)^\vee(x)\quad \text{ and } \quad
    T_{m_\phi} f(x) = (m_\phi \hat f)^\vee (x), \qquad \forall x \in \mathbb R^n. 
\end{align*}
 Hence, by Fubini's theorem,
\begin{align*}
    T_{m_\phi} f(x) &= \int_{\mathbb R^n} m_\phi(\xi)\hat f(\xi) e^{2\pi i x\cdot \xi}  \, d\xi = \int_{\mathbb R^n} \left( \int_{\mathbb R^k} m(\xi, y) \phi(y)\, dy \right)\hat f(\xi) e^{2\pi i x\cdot \xi}  \, d\xi \\ &=\int_{\mathbb R^k} \left( \int_{\mathbb R^n} m(\xi, y)  \hat f(\xi) e^{2\pi i x\cdot \xi}\, d\xi \right)  \phi(y)\, dy =
    \int_{\mathbb R^k} T_{m(\cdot, y)} f(x)  \phi(y)\, dy, 
\end{align*}
and the result follows as in Corollary \ref{cor:operators_prob_meas} together with the density of $L^{p,1}(v)$ by functions in $\mathcal C^{\infty}_c(\mathbb R^n) \cap L^{p,1}(v)$. 
\end{proof}

\medskip

\noindent
{\bf (II) Integral operators}

\medskip

\noindent Let us now consider the operator
$$
Tf(x)= \int_{\mathbb R^m} K(x,y) f(y) dy, \qquad x \in \mathbb R^n,
$$
where the integral kernel $K$ satisfies some size condition of the form $|K(x,y)|\lesssim |x-y|^{-n}$.
\begin{proposition} 

Assume that, for every $s > 0$,
$$
T_s f(x)= \int_{|x - y|\ge s} K(x, y) f(y) dy, \qquad x \in \mathbb R^n,
$$
satisfies that
$$
T_s:L^1_\mathcal{R}(u) \longrightarrow L^{1,\infty}(u), \qquad \varphi(||u||_{A_1}),\qquad \forall u \in A_1,
$$
where $\varphi$ is a positive nondecreasing function on $[1,\infty)$. Then, if $\phi$ is a bounded variation function on $(0,\infty)$ with $\lim_{x \rightarrow 0^+}\phi(x) = 0$, we have that
$$
T_\phi f(x)= \int_{\mathbb R^m} K(x, y) \phi(|x-y|) f(y) dy, \qquad x \in \mathbb R^n,
$$
satisfies that
$$
T_\phi: L^1_{\mathcal R}(u) \longrightarrow L^{1, \infty}(u),\quad C_1\varphi(C_2||u||_{A_1})(1+\log ||u||_{A_1}),\quad \forall u \in A_1. 
$$
\end{proposition} 
 
\begin{proof} We observe that, by hypothesis, 
$$
\phi(|x-y|)= \int_0^{|x-y|}  \phi'(s) ds, \qquad \phi'\in L^1(\mathbb R^n), 
$$
and hence, for every $x \in \mathbb R^n$ and every $\varepsilon>0$, by Fubini's theorem we have that \begin{align*}
    T_\phi f(x)-\phi(\varepsilon)Tf(x) &= \int_0^\infty \left(\int_{|x-y| \ge s\ge \varepsilon }K(x, y) 
f(y) dy \right) {\phi'(s) } \, ds\\ & = \int_\varepsilon^\infty T_sf(x) \phi'(s) \, ds,
\end{align*} is an average operator, and so the result follows by Corollary~\ref{cor:operators_prob_meas} and letting $\varepsilon$ tend to zero. 
\end{proof}

 \section{Limited extrapolation results}%\label{limited}
 
 The   motivation of this section comes from the fact  that there are also many operators in harmonic analysis  (such as the  Bochner-Riesz)  so that 
$$
T: L^{p_0}(v) \longrightarrow L^{p_0}(v)
$$
is not bounded for every $v\in A_{p_0}$ but is bounded for every $v$ in a certain subclass of $A_{p_0}$. Under this weaker hypothesis, only boundedness on $L^p(v)$ of $T$ can be deduced whenever $p\in (p_-, p_+)$ for certain values of $p_-$ and $p_+$. The purpose of this section is to establish some equivalence,  similar to Theorem \ref{thrm:principal_extrap_result}, between the boundedness at the endpoint $p_-$ and restricted weak-type boundedness at the $p$ level. Indeed, the Rubio de Francia extrapolation results in this case are called limited extrapolation (see \cite{am:am,cdl:cdl,coc:coc,cmpL:cmpL,d:d}).

\begin{definition} Given $0 \leq \alpha, \beta \leq 1$ and  $1 \leq p < \infty$, let us define the classes of weights 
$$
A_{p;(\alpha, \beta)} = \left\{0<v\in L^1_{\text{loc}}(\mathbb{R}^n) : v = v_0^\alpha v_1^{\beta(1-p)}, v_j \in A_1\right\} 
$$ with $$||v||_{A_{p;(\alpha,\beta)}} = \inf \left\{ ||v_0||_{A_1}^\alpha||v_1||_{A_1}^{\beta(p-1)}: v = v_0^\alpha v_1^{\beta(1-p)} \right\},$$
and
$$
\hat A_{p;(\alpha, \beta)} = \left\{0<v\in L^1_{\text{loc}}(\mathbb{R}^n) : v = v_0^\alpha (Mh)^{\beta(1-p)}, v_0 \in A_1, \ h\in L^1_\text{loc}(\mathbb R^n)\right\}
$$ with $$||v||_{\hat A_{p;(\alpha, \beta)} } = \inf \left\{||v_0||_{A_1}^{\frac\alpha{1 + \beta(p-1)}}: v = v_0^\alpha (Mf)^{\beta(1 - p)}\right\}.$$
\end{definition}

\begin{definition} Given $1 \leq p_0 < \infty$ and $0 \leq \alpha, \beta \leq 1$, set
$$
p_{+} = \frac{p_0}{1 - \alpha}, \hspace{5mm} p_{-}' = \frac{p_0'}{1-\beta}, \quad \left( \text{or } p_{-} = \frac{p_0}{1 + \beta(p_0-1)} \right),
$$ 
where $p_+ = \infty$ if $\alpha = 1$ and $p_- = 1$ if $\beta = 1$. Then, $1 \leq p_{-} \leq p_{+} \leq \infty$ and we can associate to every $p \in [p_-, p_{+}]$ the indices $$\alpha(p) = \frac{p_+ - p}{p_+} \qquad \text{ and } \qquad \beta(p) = \frac{p - p_-}{p_-(p - 1)},$$ so that $0 \leq \alpha(p), \beta(p) \leq 1$, $p_{+} = \frac{p}{1 - \alpha(p)}$, $p_{-}' = \frac{p'}{1-\beta(p)}$ and $\alpha(p_0) = \alpha$, $\beta(p_0) = \beta$. 
\end{definition}

\begin{theorem}[\cite{d:d}]\label{thrm:limited_extrapolation_Ap}
Let $(f, g)$ be a pair of measurable functions such that for some $1 \leq p_0 < \infty$ and $0 \leq \alpha , \beta \leq 1$ (not both identically zero) we have $$\norm{g}_{L^{p_0}(v)} \leq \varphi\left(\norm{v}_{A_{p_0;(\alpha,\beta)}}\right) \norm{f}_{L^{p_0}(v)}, \qquad \forall v \in A_{p_0;(\alpha, \beta)},$$ where $\varphi$ is a nondecreasing function on $[1,\infty)$. Then, for every $p_- < p < p_{+}$, $$\norm{g}_{L^{p}(v)} \leq C_1\varphi\parenthesis{C_2 \norm{v}_{A_{p;(\alpha(p),\beta(p))}}^{\max\parenthesis{\frac{p_+ - p_0}{p_+ - p}, \frac{p_0 - p_-}{p - p_-}} } } \norm{f}_{L^{p}(v)}, \qquad \forall v \in A_{p;(\alpha(p), \beta(p))}, $$with $C_1$ and $C_2$ being two positive constants independent of all parameters involved. 
\end{theorem}

Observe that in Theorem~\ref{thrm:limited_extrapolation_Ap} is not possible to extrapolate till the endpoints $p_-$ and $p_+$. However, in \cite[Theorem~3.7]{coc:coc} the authors were able to obtain an estimate in the endpoint $p_-$. To do so, they needed to assume that the operators satisfy a restricted weak-type boundedness for the class of weights $\hat A_{p;(\alpha, \beta)}$  which is a slightly bigger class than $A_{p;(\alpha, \beta)}$. 

\begin{theorem}[\cite{coc:coc}] 
\noindent Let $1 \leq p_0 < \infty$, $0 \leq \alpha, \beta \leq 1$ (not both identically zero) and let $T$ be an operator. Assume that \begin{equation*}
    T:L^{p_0, 1}(v) \longrightarrow L^{p_0, \infty}(v), \qquad \varphi(\Vert v\Vert _{\hat A_{p_0;(\alpha, \beta)}}), \qquad \forall v \in \hat A_{p_0;(\alpha, \beta)},
\end{equation*}
where $\varphi$ is a positive nondecreasing function on $[1, \infty)$. Then:

\begin{enumerate}[(i)]
    \item If $p_- > 1$, 
    \begin{equation*}%\label{eq:conseq_i_FromhatApalphabetatop-}
        T:L^{p_-,1}\left(u^{\alpha(p_-)}\right) \longrightarrow L^{p_-, \infty}\left(u^{\alpha(p_-)}\right), \quad \frac {\Phi_{p_-}(||u||_{A_1}^{\alpha(p_-)})}{p_- - 1}, \quad \forall u \in A_1,
    \end{equation*} 
    where $\Phi_{p_-}$ is a positive nondecreasing function on $[1, \infty)$.
    
    \item If $p_- = 1$,  
    \begin{equation*}%\label{eq:conseq_ii_FromhatApalphabetatop-}
    T:L^1_\mathcal R\left(u^{\alpha(p_-)}\right) \longrightarrow L^{1, \infty}\left(u^{\alpha(p_-)}\right), \qquad \Phi_1(||u||_{A_1}^{\alpha(p_-)}), \qquad \forall u \in A_1.
    \end{equation*}
     
\end{enumerate}

\end{theorem}

Our following theorem shows that the converse is also true:

\begin{theorem}\label{thrm:restricted_extrapolation_limited_lorentz}

\noindent Let $(f, g)$ be a pair of measurable functions such that for some $1 \leq p_0 < \infty$ and $0 < \alpha \leq 1$, 
\begin{equation*}%\label{thrm:restricted_extrapolation_limited_lorentz_hyp}
    ||g||_{L^{p_0,\infty}(u^\alpha)} \leq \varphi\parenthesis{\norm{u}_{A_1}^\alpha} ||f||_{L^{p_0,1}(u^\alpha)}, \qquad \forall u \in A_1,
\end{equation*} 
with $\varphi$ being a nondecreasing function on $[1,\infty)$. Then, for any $p_0 \leq p < \frac{p_0}{1 - \alpha}$, $$||g||_{L^{p,\infty}(v)} \leq \Psi\left(||v||_{\hat A_{p;(\alpha(p), \beta(p))}}\right) ||f||_{L^{p, 1}(v)}, \qquad \forall v \in \hat{A}_{p;(\alpha(p), \beta(p))},$$ where $\alpha(p) = 1 - \frac{p(1 - \alpha)}{p_0}$, $\beta(p) = \frac{p - p_0}{p_0(p-1)}$ and, for every $r \geq 1$,
\begin{equation*}
    \Psi(r) = C_1\left(\frac 1{p_0 - p(1 - \alpha)}\right)^{\frac{p - p_0}p}\varphi\left(C_2 r^\frac{\alpha p}{p_0 - p(1 - \alpha)}\right) r^{\frac{\alpha (p - p_0)}{p_0 - p(1 - \alpha)}} \left(1 + \log r\right)^{\frac{2(p - p_0)}p},
\end{equation*} with $C_1$ and $C_2$ being two positive constants independent of all parameters involved.

\end{theorem}

\begin{proof}

\noindent Let $h \in L^1_{\text{loc}}(\mathbb R^n)$ and $u \in A_1$ so that $$v = (Mh)^{\beta(p)(1-p)} u^{\alpha(p)}\in \widehat A_{p;(\alpha(p),\beta(p))}.$$ Further, take $t = 1 + \frac 1{2^{n + 1}\lVert u \rVert_{A_1}}$ so that $u^t\in A_1$ with $||u^t||_{A_1} \lesssim ||u||_{A_1}$ (see \eqref{eq:u^t_A1}) and, since $t > 1$, $$ \frac {p - p_0}p = \frac{\alpha - \alpha(p)}{1 - \alpha(p)} < \frac{t\alpha - \alpha(p)}{t - \alpha(p)} < \alpha.$$ Hence, we can take $$v_\theta= (Mh)^{\beta(1-p)} u^{\alpha(p)\theta} \qquad \text{ with }
\qquad \frac{p - p_0}p <\theta < \frac{t\alpha - \alpha(p)}{t - \alpha(p)}. $$
Besides, since $\alpha(p) \leq \alpha$, letting $$\mu= 1- \frac{\alpha(p)(1-\theta)}{\alpha t} \in (0,1),$$  then $\theta < \alpha \mu < 1$ and, by \eqref{eq:MfumuA1}, for every measurable set $F \subseteq \mathbb R^n$, $$
u_0 = \left(M_{\alpha\mu}( \chi_F v_\theta) u^{\alpha(p)(1-\theta)}\right)^{\frac1\alpha} = M\left( \chi_F v_\theta^{\frac1{\alpha\mu}} \right)^\mu (u^t)^{1-\mu} \in A_1,$$ with $ ||u_0||_{A_1}\leq \frac{C||u||_{A_1}}{1-\mu}$. 

Now, let $y > 0$ and set $F=\{x : |g(x)|>y\}$ so that $v(F) = \lambda^v_g(y)$. We can assume, as it was done in the proof of Theorem \ref{thrm:principal_extrap_result}, that $v(F)<\infty$. Then,  by hypothesis,  we obtain that
\begin{align*}
&y^{p_0}\lambda_g^v(y) = y^{p_0}\int_{\{x \,:\,|g(x)|>y\}} v(x)\,dx \le y^{p_0} \int_F u_0(x)^\alpha  dx
\\
&\leq  \varphi \left(||u_0||_{A_1}^\alpha\right)^{p_0}\left[ p_0\int_0^\infty \left( \int_{\{|f(x)| > z\}} u_0(x)^\alpha \,dx  \right)^\frac 1{p_0}\,dz\right]^{p_0}
\\
&\lesssim  \varphi \left(||u_0||_{A_1}^\alpha\right)^{p_0} \left \Vert \frac{M_{\alpha \mu}( \chi_F v_\theta) }{v_\theta}  \right \Vert_{L^{\left(\frac p{p_0}\right)',\infty}(v)} \left[ \int_0^\infty ||\chi_{\{|f(x)| > z\}} ||_{L^{\frac p{p_0}, 1}(v)}^\frac 1{p_0}\,dz\right]^{p_0} 
\\
&\lesssim \varphi \left(\left[\frac{C\alpha t||u||_{A_1}}{\alpha(p)(1-\theta)}\right]^\alpha\right)^{p_0} \left \Vert \frac{M_{\alpha \mu}( \chi_F v_\theta) }{v_\theta}  \right \Vert_{L^{\left(\frac p{p_0}\right)',\infty}(v)} ||f||_{L^{p,1}(v)}^{p_0},
\end{align*}

\noindent where in the penultimate estimate we have used the H\"older's inequality for Lorentz spaces with respect to the measure $v(x)\,dx$.

Now, since $\beta(p)(1 - p) = 1 - \frac p{p_0}$, then $v \in \hat A_{\frac p{p_0}}$ and, by virtue of Lemma \ref{difi},

\begin{align*}
\left \Vert \frac{M_{\alpha\mu}( \chi_F v_\theta) }{v_\theta}  \right \Vert_{L^{\left(\frac p{p_0}\right)',\infty}(v)} &\lesssim  C_{\frac p{p_0}, \theta, \alpha\mu}(u) v(F)^\frac{p - p_0}p = C_{\frac p{p_0}, \theta, \alpha\mu}(u) \lambda^v_g(y)^\frac{p - p_0}p,
\end{align*}

\noindent so taking the supremum over all $y > 0$, in particular, we obtain that $$ \lVert g \rVert_{L^{p,\infty}(v)} \lesssim C_{\frac p{p_0}, \theta, \alpha\mu}(u)^{\frac1{p_0}} \varphi \left(\frac{\tilde C ||u||_{A_1}^\alpha }{(1-\theta)^\alpha}\right)  ||f||_{L^{p,1}(v)}.$$

Finally, concerning about the constant $C_{\frac p{p_0}, \theta, \alpha\mu}(u)$, we observe that
\begin{align*}
    C_{\frac p{p_0}, \theta, \alpha\mu}(u) & = \left(\frac {p^2}{(p-p_0)^2(\alpha\mu-\theta) (\theta-\frac{p-p_0}p)^2}\right)^{\theta}\Vert u\Vert_{ A_1}^{2\theta - \frac {2(p - p_0)}{p_0}} \\ & \lesssim \left(\frac {p^2}{(p-p_0)^2(\alpha-\theta) (\theta-\frac{p-p_0}p)^2}\right)^{\theta} \Vert u\Vert_{ A_1}^{3\theta - \frac {2(p - p_0)}{p_0}},
\end{align*}

\noindent so the behaviour of the constant $C_{\frac p{p_0}, \theta, \alpha\mu}(u)$ follows as in the proof of Theorem~\ref{thrm:principal_extrap_result}. 

\end{proof}

As an application, we present some new weighted estimates for the Bochner-Riesz operator below the critical index. 

\begin{proposition}[\cite{KrLm:KrLm}] \label{n=2} Let $n=2$ and $0<\lambda<\frac12$. Then,
$$
B_\lambda: L^{\frac 4{3 + 2\lambda}}\left(u^\frac{2\lambda}{3 + 2\lambda}\right) \longrightarrow L^{\frac 4{3 + 2\lambda}, \infty}\left(u^\frac{2\lambda}{3 + 2\lambda} \right),\quad c(n,\lambda) \norm{u}^{\frac{\lambda(7 + 4\lambda)}{6 + 4\lambda}}_{A_1}, \quad  \forall u \in A_1.
$$
\end{proposition}

%Further, for $n > 2$, we have the following result.
\begin{proposition}[\cite{Cm:Cm3}]
\label{n>2}
Let $n > 2$ and $\frac{n-1}{2(n+1)} < \lambda < \frac{n-1}2$. Then
$$
B_{\lambda}: L^2\left(u^\frac{1 + 2\lambda}n\right) \longrightarrow L^2\left(u^\frac{1 + 2\lambda}n\right), \qquad \varphi\left(||u||_{A_1}^\frac{1 + 2\lambda}n\right), \qquad \forall u \in A_1,
$$
where $\varphi$ is a positive nondecreasing function on $[1,\infty)$.

\end{proposition}

Therefore, as a consequence of Theorem~\ref{thrm:restricted_extrapolation_limited_lorentz} and Propositions~\ref{n=2} and \ref{n>2}, we obtain the following result.

\begin{corollary}

Let $n = 2$ and $0 < \lambda < \frac 12$. For every $\frac 4{3 + 2\lambda} \leq p < \frac 43$,
$$
B_\lambda: L^{p,1}(v) \longrightarrow L^{p, \infty}(v), \qquad  \forall v \in \hat A_{p;\left(\frac{4 - 3p}4, \frac{(3 + 2\lambda)p - 4}{4(p-1)}\right)}.
$$
\noindent Now, let $n > 2$ and $\frac{n-1}{2(n+1)} < \lambda < \frac{n-1}2$. For every $2 \leq p < \frac{2n}{n - 1 - 2\lambda}$,
$$
B_{\lambda}: L^{p,1}(v) \longrightarrow L^{p,\infty}(v), \qquad \forall v \in \hat A_{p;\parenthesis{\frac{2n - p(n - 1 - 2\lambda)}{2n},\frac{p-2}{2(p-1)}}}.
$$
\end{corollary}

\end{document}